%% file: rank16.tex
\documentclass[english]{smfart}
\include{preamble}

\usepackage{bull}

\newcommand{\lcm}{\operatorname{LCM}}
\newcommand{\rk}{\operatorname{rk}}

\newcommand{\nfsqr}{(n + \lfloor \sqrt{2n} \rfloor)}

\author{Martin Orr}
\address{Université Paris-Sud\\ Bat.\ 425\\ 91400 Orsay\\ France}
\email{martin.orr@math.u-psud.fr}

\title{Lower bounds for ranks of Mumford--Tate~groups}
\alttitle{Minoration des rangs de groupes de Mumford--Tate}

\begin{document}

\frontmatter

\begin{abstract}
Let $A$ be a complex abelian variety and $G$ its Mumford--Tate group.
Supposing that the simple abelian subvarieties of $A$ are pairwise non-isogenous,
we find a lower bound for the rank $\rk G$ of $G$, which is a little less than $\log_2 \dim A$.
If we suppose that $\End A$ is commutative, then we show that $\rk G \geq \log_2 \dim A + 2$,
and this latter bound is sharp.
We also obtain the same results for the rank of the $\ell$-adic monodromy group of an abelian variety defined over a number field.
\end{abstract}

\begin{altabstract}
Soit $A$ une variété abélienne complexe et $G$ son groupe de Mumford--Tate.
En supposant que les sous variétés abéliennes simples de $A$ sont deux à deux non-isogènes,
on trouve une minoration du rang $\rk G$ de $G$, légèrement inférieure à $\log_2 \dim A$.
Si on suppose que $\End A$ est commutatif, alors on montre que $\rk G \geq \log_2 \dim A + 2$,
et cette borne-ci est optimale.
On obtient les mêmes resultats pour le rang du groupe de monodromie $\ell$-adique d'une variété abélienne définie sur un corps de nombres.
\end{altabstract}


\maketitle

\mainmatter

\section{Introduction}

Let $A$ be a complex abelian variety of dimension $g$, whose simple abelian subvarieties are pairwise non-isogenous.
In this paper we will establish a lower bound for the rank of the Mumford--Tate group of $A$.
The Mumford--Tate group is an algebraic group over $\Q$ defined via the Hodge theory of $A$
(see section~\ref{sec:mt-group-definition} below for the definition).
The same argument will also establish a lower bound for the rank of the $\ell$-adic monodromy groups $G_\ell$,
in the case where $A$ is defined over a number field.
The $\ell$-adic monodromy group is the Zariski closure of the image of the Galois representation on the $\ell$-adic Tate module of $A$.
Our main theorems are the following:

\begin{theorem} \label{thm:main-bound-commutative}
Let $A$ be an abelian variety of dimension $g$ such that $\End A$ is commutative.
Let $G$ be the Mumford--Tate group or the $\ell$-adic monodromy group of $A$.
Then $\rk G \geq \log_2 g + 2$.
\end{theorem}

\begin{theorem} \label{thm:main-bound-noncommutative}
Let $A$ be an abelian variety of dimension $g$ whose simple abelian subvarieties are pairwise non-isogenous.
Let $G$ be the Mumford--Tate group or the $\ell$-adic monodromy group of $A$.
If $n = \rk G$, then
\[ n + \alpha(n)\sqrt{n \log_e n} \geq \log_2 g + 2 \]
for a function $\alpha : \N_{\geq 2} \to \R$ satisfying $\alpha(n) < 2$ for all $n$ and $\alpha(n) \to 1/\log_e 2 = 1.44...$ as $n \to \infty$.
\end{theorem}

Each of these theorems is an instance of a more general bound for weak Mumford--Tate triples, which are defined in section~\ref{sec:mt-group-definition}.
These more general bounds are Theorems~\ref{thm:mt-triple-bound-commutative} and~\ref{thm:mt-triple-bound-noncommutative} respectively.
These would apply also for example to the analogue of the Mumford--Tate group for a Hodge--Tate module of weights~0 and~1.

Theorem~\ref{thm:main-bound-commutative} was proved by Ribet in the case of an abelian variety with complex multiplication~\cite{ribet:cm}.
Our proof is a generalisation of his, relying on the fact that the defining representation of the Mumford--Tate group or $\ell$-adic monodromy group has minuscule weights.

\vskip 1em

The condition on simple subvarieties in Theorem~\ref{thm:main-bound-noncommutative} is necessary:
taking products of copies of the same simple abelian variety increases the dimension without changing the rank of the Mumford--Tate group.
Indeed, if $A$ is isogenous to $\prod_i A_i^{m_i}$ where the $A_i$ are simple and pairwise non-isogenous, then according to \cite{hindry-ratazzi:product-mt} Lemme~2.2,
\[ \operatorname{MT}(A) \cong \operatorname{MT}(\prod_i A_i). \]
Hence Theorem~\ref{thm:main-bound-noncommutative} implies that for a general abelian variety $A$, if $n$ denotes the rank of either the Mumford--Tate group or the $\ell$-adic monodromy group of~$A$, then
\[ n + \alpha(n)\sqrt{n \log_e n} \geq \log_2 \left( \sum_i \dim A_i \right) + 2 \]
where the $A_i$ are one representative of each isogeny class of simple abelian subvarieties of $A$.

The condition of having pairwise non-isogenous simple abelian subvarieties can be interpreted via the endomorphism algebra like the condition in Theorem~\ref{thm:main-bound-commutative}:
it is equivalent to $\End A \otimes_\Z \Q$ being a product of division algebras.
Note also that $\End A$ being commutative implies the condition of Theorem~\ref{thm:main-bound-noncommutative}.
(Throughout this paper, $\End A$ means the endomorphisms of $A$ after extension of scalars to an algebraically closed field.)

\vskip 1em

Let $G$ be either the Mumford--Tate group or the $\ell$-adic monodromy group of $A$.
It is well known that the rank of $G$ is at most $g + 1$,
and that this upper bound is achieved for a generic abelian variety.
Indeed, if $g$ is odd and $\End A = \Z$, then $\rk G$ is always $g+1$~\cite{serre:rank-odd-dim}.
So in this case the bound in Theorem~\ref{thm:main-bound-commutative} is far from sharp.

On the other hand if $g$ is a power of $2$,
then there are abelian varieties for which the bound in Theorem~\ref{thm:main-bound-commutative} is achieved
(even with $\End A = \mathbb{Z}$).
We construct such examples in section~\ref{sec:examples}.
The exact bound for a given $g$ is very sensitive to the prime factors of $g$.
Equality can happen only when $g$ is a power of $2$ (for the trivial reason that otherwise $\log_2 g \not\in \Z$) but even near-equality can only occur when $g$ has many small prime factors.
This was made precise by Dodson in the complex multiplication case \cite{dodson:cm},
and it is possible that something similar could be proved in general.

Theorem~\ref{thm:main-bound-noncommutative} is not sharp.
The function $\alpha(n)$ is specified exactly in section~\ref{sec:multiplicity},
but it is likely that this could be improved on, perhaps to something which goes to $0$ as $n \to \infty$.
In section~\ref{sec:examples}, we construct a family of examples showing that Theorem~\ref{thm:main-bound-noncommutative} cannot be improved to $n + k \geq \log_2 g$ for any constant $k$.

\vskip 1em

We can deduce a lower bound for the growth of the degrees of the division fields $K(A[\ell^n])$ (for $\ell$ a fixed prime number) as a straightforward consequence of Theorem~\ref{thm:main-bound-commutative}.

\begin{corollary}
Let $A$ be an abelian variety of dimension $g$ over a number field~$K$, and $\ell$ a prime number.
If $\End A$ is commutative,
then there is a constant $C(A, K, \ell)$ such that
\[ [ K(A[\ell^n]) : K ] \geq C(A, K, \ell) \, \ell^{n (\log_2 g + 2)}. \]
\end{corollary}

Theorem~\ref{thm:main-bound-noncommutative} implies a similar bound for the degree of $K(A[\ell^n])$ whenever $A$ is an abelian variety whose simple abelian subvarieties are pairwise non-isogenous.
One would like to extend these results to lower bounds on the degrees of $K(A[N])$ for any $N$,
but this cannot be done without knowing how $C(A, K, \ell)$ varies with $\ell$.
The primary obstacle here is the index of the image of $\Galabs{K}$ in $G_\ell(\Z_\ell)$,
which is conjectured to be bounded by a constant $C_1(A, K)$ independent of $\ell$.

\vskip 1em

In section~\ref{sec:mt-group-definition} we recall the definitions of Mumford--Tate group, $\ell$-adic monodromy group and weak Mumford--Tate triple, an axiomatisation of the properties of the groups and representations we will consider.
In section~\ref{sec:counting-chars} we bound the number of distinct characters of a maximal torus which can appear in such a representation.
In section~\ref{sec:multiplicity} we bound the multiplicity of absolutely irreducible components of this representation.
This is straightforward for the Mumford--Tate group but more difficult for the $\ell$-adic monodromy group.
Combining these two bounds gives Theorems~\ref{thm:main-bound-commutative} and~Theorem~\ref{thm:main-bound-noncommutative}.
Finally in section~\ref{sec:examples} we give some examples to show that Theorem~\ref{thm:main-bound-commutative} is sharp and to place a limit on the possible improvements of Theorem~\ref{thm:main-bound-noncommutative}.

\paragraph*{Acknowledgements.}
I am grateful to Emmanuel Ullmo who suggested to me the problem treated in this paper, and for regular conversations during its preparation.
I would also like to thank Barinder Banwait for his comments on an early version of the manuscript,
and the referee for their careful attention to detail.

\section{Mumford--Tate triples: Definitions} \label{sec:mt-group-definition}

We recall the definition of a weak Mumford--Tate triple, which abstracts the key properties of a Mumford--Tate group which we will use.
We recall also the definitions of the two examples of Mumford--Tate triple we will consider, namely the Mumford--Tate group and the $\ell$-adic monodromy group of an abelian variety.

The following definition is a slight modification of those used by Serre~\cite{serre:minuscule} and Wintenberger~\cite{wintenberger:mt-reps}.

\begin{definition}
Let $F$ be a field of characteristic zero and $E$ an algebraically closed field containing $F$.

A \defterm{weak Mumford--Tate triple} is a triple $(G, \rho, \Psi)$ where $G$ is an algebraic group over $F$, $\rho$ is a rational representation of $G$ and $\Psi$ is a set of cocharacters of $G \times_F E$ satisfying the following conditions:
\begin{enumerate}[(i)]
\item $G$ is a connected reductive group;
\item $\rho$ is faithful;
\item the images of all $G(E)$-conjugates of elements of $\Psi$ generate $G_E$.
\end{enumerate}

The \defterm{weights} of a Mumford--Tate triple $(G, \rho, \Psi)$ are the integers which appear as weights of $\rho \circ \nu$ (a representation of $\G_m$) for some $\nu \in \Psi$.

A weak Mumford--Tate triple $(G, \rho, \Psi)$ is called \defterm{pure} if $\rho(G)$ contains the torus $\G_m.\id$ of homotheties.
\end{definition}

\paragraph*{The Mumford--Tate group}

Let $A$ be an abelian variety over $\C$, of dimension~$g$.
The singular cohomology group $H^1(A(\C), \Q)$ is a vector space of dimension $2g$ over $\Q$.
Hodge theory gives a decomposition of $\C$-vector spaces
\[ H^1(A(\C), \Q) \otimes_\Q \C = H^{1,0}(A) \oplus H^{0,1}(A) \]
with $H^{1,0}(A)$ and $H^{0,1}(A)$ being mapped onto each other by complex conjugation
(so each has dimension $g$).

We define a cocharacter $\mu : \G_{m,\C} \to \GL_{2g,\C}$ by:
\begin{align*}
\mu(z) \,	& \text{acts as multiplication by } z \text{ on } H^{1,0}(A)	\\
		& \text{and as the identity on } H^{0,1}(A).
\end{align*}
The \defterm{Mumford--Tate group} of $A$ is defined to be the smallest algebraic subgroup $M$ of $\GL_{2g}$ defined over $\Q$ and such that $M_\C$ contains the image of $\mu$.

The triple consisting of the Mumford--Tate group, its defining representation $\rho : M \to \GL_{2g}$,
and the set of $\Aut(\C/\Q)$-conjugates of the cocharacter $\mu$ form a pure weak Mumford--Tate triple of weights $\zeroone$.
This is immediate from the definitions.

The functor $A \mapsto H^1(A(\C), \Z)$ is an equivalence of categories between complex abelian varieties and polarisable $\Z$-Hodge structures of type $\{ (-1,0), (0,-1) \}$.
Furthermore the endomorphism ring of $\rho$ as a representation of the Mumford--Tate group is equal to the endomorphism ring of $H^1(A(\C), \Q)$ as a $\Q$-Hodge structure, so
\[ \End \rho = \End A \otimes_\Z \Q. \]

\paragraph*{The $\ell$-adic algebraic monodromy group}
Now suppose that the abelian variety $A$ is defined over a number field $K$.
Its first $\ell$-adic cohomology group is a $\Q_\ell$-vector space of dimension $2g$,
isomorphic to the dual of the $\ell$-adic Tate module:
\[ H^1(A_{\bar{K}}, \Q_\ell) \isom \left( T_{\ell} A \otimes_{\Z_\ell} \Q_\ell \right)^\vee. \]

The Galois group $\Galabs{K}$ acts on the torsion points of $A(\bar{K})$, and this induces an action on $H^1(A_{\bar{K}}, \Q_\ell)$, or in other words a continuous representation
\[ \rho_\ell : \Galabs{K} \to \GL_{2g}(\Q_\ell). \]

The \defterm{$\ell$-adic algebraic monodromy group} of $A$ is the smallest algebraic subgroup $G_\ell$ of $\GL_{2g,\Q_\ell}$ whose $\Q_\ell$-points contain the image of $\rho_\ell$.
By working with the $\ell$-adic monodromy group instead of the image of $\rho_\ell$ directly,
we gain the ability to use the structure theory of algebraic groups.
On the other hand, we do not lose very much because $\Im \rho_\ell$ is known \cite{bogomolov:open-image} to be an open (and hence finite-index) subgroup of $G_\ell(\Q_\ell) \cap \GL_{2g}(\Z_\ell)$.

Pink~\cite{pink:mt} has proved that the identity component $G_\ell^\circ$ together with the representation~$\rho_\ell$ and a certain set $\Psi$ of cocharacters form a pure weak Mumford--Tate triple of weights~$\zeroone$.

By Faltings' Theorem~\cite{faltings:endomorphisms},
\[ \End \rho_\ell = \End A \otimes_\Z \Q_\ell. \]

\section{Bound for the number of characters} \label{sec:counting-chars}

Let $(G, \rho, \Psi)$ be a pure weak Mumford--Tate triple of weights~$\zeroone$,
and let $T$ be a maximal torus of $G$.
In this section we will give an upper bound for the number of distinct characters in $\rho_{|T}$ as a function of $\rk G$.

If $A$ has complex multiplication (in other words if $G$ is a torus) then this bound was obtained by Ribet~\cite{ribet:cm}.
Our method of proving the bound is inspired by applying Ribet's method to a maximal torus of $G$, but it is convenient to arrange it differently.

\begin{proposition} \label{prop:mt-weight-count}
Let $(G, \rho, \Psi)$ be a pure weak Mumford--Tate triple of weights $\zeroone$.
The number of distinct characters in $\rho_{|T}$ is at most~$2^{\rk G - 1}$.
\end{proposition}

\begin{proof}
Let $Y = \Hom(\G_{m, E}, T_E) \otimes_\Z \Q$ be the quasi-cocharacter space of $T$, where $E$ is an algebraically closed field of definition for $(G, \rho, \Psi)$.

Let $\Psi'$ be the set of all cocharacters of $T_E$ which are $G(E)$-conjugate to an element of $\Psi$.
Every cocharacter of $G$ has a $G(E)$-conjugate whose image is contained in $T_E$, so $\Psi'$ still satisfies condition (iii) in the definition of a weak Mumford--Tate triple.
Replacing $\Psi$ by $\Psi'$ does not change the weights of our Mumford--Tate triple.

Furthermore $\Psi'$ is closed under the action of the Weyl group of $G_E$ on $Y$.
So condition (iii) implies that $\Psi'$ spans $Y$ as a $\Q$-vector space.

Let $\Theta$ be a basis of $Y$ contained in $\Psi'$.
The character space of $T$ is dual to $Y$, so any character $\omega$ is determined by its inner products
$\innerprod{\omega}{\mu}$ for $\mu \in \Theta$.

Because our Mumford--Tate triple has weights~$\zeroone$, if $\mu$ is a character in $\rho_{|T}$ then these inner products can only have the values $0$ or $1$.
So there are at most $2^{\abs{\Theta}}$ distinct characters in $\rho_{|T}$, and $\abs{\Theta} = \rk G$.

We can use the fact that our Mumford--Tate triple is pure to improve the exponent to $\rk G - 1$.
We know that $\rho(G)$ contains the homotheties.
Since $\rho$ is faithful, there is a unique cocharacter $\mu_0 : \G_m \to G$ such that $\rho \circ \mu_0(z) = z.\id$.
We take $\Theta'$ to be a subset of $\Psi$ such that $\Theta' \cup \{ \mu_0 \}$ is a basis of $Y$.
Now $\innerprod{\omega}{\mu_0} = 1$ for all characters $\omega$ in $\rho_{|T}$,
so $\omega$ is determined by the values $\innerprod{\omega}{\mu}$ for $\mu \in \Theta'$.
We may repeat the previous argument with $\Theta$ replaced by~$\Theta'$.
\end{proof}

\begin{corollary} \label{cor:product-bound}
Let $(G, \rho, \Psi)$ be a pure weak Mumford--Tate triple of weights $\zeroone$.
Let $M$ be the maximum of the multiplicities of the irreducible components of $\rho$ (working over an algebraically closed base field).
Then
\[ \dim \rho \leq M \cdot 2^{\rk G - 1}. \]
\end{corollary}

\begin{proof}

Serre~\cite{serre:minuscule} showed that each irreducible component $\sigma$ in a weak Mumford--Tate triple of weights $\{ 0, 1\}$ is \defterm{minuscule},
that is, the characters in $\sigma_{|T}$ form a single orbit under the action of the Weyl group.
Serre only treated strong Mumford--Tate triples, i.e. weak Mumford--Tate triples satisfying the additional condition that all the cocharacters in $\Psi$ are contained in a single $\Aut(E/F)$-orbit.
However this extra condition is not used in his argument (see also \cite{pink:mt} Section~4 and~\cite{zarhin:minuscule}).

The characters of $T$ in a minuscule representation have multiplicity $1$, and non-isomorphic minuscule representations contain disjoint characters.
So the multiplicity of any character in $\rho_{|T}$ is equal to the multiplicity of the unique irreducible component which contains that character, and so
\[ \dim \rho \leq M \cdot \bigl( \text{the number of distinct characters in } \rho_{|T} \bigr). \]

The corollary now follows from Proposition~\ref{prop:mt-weight-count}.
\end{proof}

\section{Bound for the multiplicities} \label{sec:multiplicity}

Let $(G, \rho, \Psi)$ be a pure weak Mumford--Tate triple of weights $\zeroone$.
In this section we will bound the multiplicities of the absolutely irreducible components of $\rho \otimes_F \bar{F}$.
If $\End \rho$ is commutative, then it is immediate that all absolutely irreducible components of $\rho \otimes_F \bar{F}$ have multiplicity $1$.

Most of the section concerns the case in which the irreducible components of $\rho$ are pairwise non-isomorphic.
Because we use a result on division algebras coming from class field theory, we must assume that the field of definition of $\rho$ is a local field or a number field.
If $n = \rk G$, then each absolutely irreducible component has multiplicity at most $\alpha(n) \sqrt{n \log_e n}$ for a function $\alpha(n)$ satisfying the conditions of Theorem~\ref{thm:main-bound-noncommutative}.

To establish this bound, we introduce an invariant $u(G)$ for a reductive group $G$ such that for any $F$-irreducible representation of $G$, the multiplicity of its irreducible components over $\bar{F}$ is at most $u(G)$.
Then we use Landau's function (the maximum LCM of a set of positive integers with given sum) to obtain a bound for $u(G)$.

The above bounds together with Corollary~\ref{cor:product-bound} suffice to prove Theorem~\ref{thm:main-bound-commutative} for both the Mumford--Tate group and $\ell$-adic monodromy groups, and Theorem~\ref{thm:main-bound-noncommutative} for the Mumford--Tate group.
Proving Theorem~\ref{thm:main-bound-noncommutative} for the $\ell$-adic monodromy group requires additional work because even when an abelian variety satisfies the condition of Theorem~\ref{thm:main-bound-noncommutative}, its associated $\ell$-adic representations might not satisfy the corresponding condition of Theorem~\ref{thm:mt-triple-bound-noncommutative}.

\subsection{The commutative endomorphism case}

\begin{theorem} \label{thm:mt-triple-bound-commutative}
Let $(G, \rho, \Psi)$ be a pure weak Mumford--Tate triple of weights $\zeroone$.
If $\End \rho$ is commutative, then $\rk G \geq \log_2 \dim \rho + 1$.
\end{theorem}

\begin{proof}
Let $F$ be the field of definition of $\rho$.
Since $\End \rho$ is commutative, each irreducible component of $\rho \otimes_F \bar{F}$ has multiplicity $1$.
So the theorem follows immediately from Corollary~\ref{cor:product-bound}.
\end{proof}

Let $A$ be an abelian variety, $G$ its Mumford--Tate group or $\ell$-adic monodromy group, and $\rho$ the associated representation.
We have observed that $\End \rho = \End A \otimes_\Z F$ where $F = \Q$ or $\Q_\ell$ as appropriate, so that if $\End A$ is commutative the same is true of $\End \rho$.
Hence Theorem~\ref{thm:main-bound-commutative} follows from Theorem~\ref{thm:mt-triple-bound-commutative}.
The $\log_2 \dim \rho +1$ becomes $\log_2 \dim A +2$ because $\dim \rho = 2\dim A$.

\subsection{Multiplicity of irreducible representations and $u(G)$}

\begin{definition}
Let $G$ be a reductive group defined over the field $F$.
Let $T$ be a maximal torus of $G$ and $\Lambda = \Hom(T_{\bar{F}}, \G_m)$ the character group of $T$.
Let $\Lambda_0$ be the subgroup of $\Lambda$ generated by the roots of $G$ and characters which vanish on $T \cap G^\der$.
The roots of $G$ span the quasi-character space of $T_{\bar{F}} \cap G_{\bar{F}}^\der$ as a $\Q$-vector space so $\Lambda_0$ spans $\Lambda \otimes_\Z \Q$.
It follows that $\Lambda/\Lambda_0$ is finite.
(In fact $\Lambda/\Lambda_0$ is canonically isomorphic to the dual of the centre of $G^\der(\bar{F})$, which is a finite abelian group.)

Hence we can define $u(G)$ to be the exponent of $\Lambda/\Lambda_0$.
\end{definition}

\begin{lemma} \label{lem:tits-brauer-order}
Let $G$ be a reductive group over a field $F$ and $\rho$ an $F$-irreducible representation of $G$.
Let $D$ be the endomorphism ring of $\rho$ and $E$ the centre of $D$.
Then the order of $[D]$ in $\Br E$ divides $u(G)$.
\end{lemma}

\begin{proof}
Fix a base $\Delta$ for the root system of $G$ with respect to $T$.
When we refer to the action of $\Galabs{F}$ on the character group $\Lambda$ below,
this is the natural action twisted by the Weyl group so that it preserves the set $\Delta$
(this is the same action used in~\cite{tits:irred-reps}).

Let $\sigma$ be an absolutely irreducible component of $\rho \otimes_F \bar{F}$,
and $\lambda_\sigma \in \Lambda$ the highest weight of $\sigma$.
Let $\Gamma$ be the subgroup of $\Galabs{F}$ fixing $\lambda$.
Then $E$ is isomorphic to the subfield of $\bar{F}$ fixed by~$\Gamma$.

Tits defined a map
\[ \alpha_{G,E} : \Lambda^\Gamma \to \Br E \]
as follows:
if $\lambda \in \Lambda^\Gamma$ is dominant then there is a unique isomorphism class of $E$-irreducible representations of $G$ with highest weight $\lambda$.
The endomorphism ring of such a representation is a division algebra with centre $E$.
We define $\alpha_{G,E}(\lambda)$ to be the inverse of the class of this division algebra in $\Br E$.
Tits showed that this map on dominant weights is additive so it extends to a homomorphism $\Lambda^\Gamma \to \Br E$.
He also showed that $\alpha_{G,E}$ is trivial on $\Lambda_0^\Gamma$ (\cite{tits:irred-reps} Corollary~3.5).

In our case we have $[D]^{-1} = \alpha_{G,E}(\lambda_\sigma)$.
Since $[D]$ is in the image of $\alpha_{G,E}$, it follows that the order of $[D]$ in $\Br E$ divides the exponent of $\Lambda^\Gamma/\Lambda_0^\Gamma$.
But the latter is a subgroup of $\Lambda/\Lambda_0$, so its exponent divides $u(G)$.
\end{proof}

\begin{corollary} \label{cor:irred-rep-multiplicity}
Let $G$ be a reductive group defined over a number field or a local field $F$.
Let $\rho$ be an $F$-irreducible representation of $G$.
Then the multiplicity of each absolutely irreducible component of $\rho \otimes_F \bar{F}$ divides $u(G)$.
\end{corollary}

\begin{proof}
Let $D = \End \rho$ and let $E$ be the centre of $D$.
Then the multiplicity of any absolutely irreducible component of $\rho \otimes_F \bar{F}$ is
$\sqrt{\dim_{E} D}$.

Since $F$ is a number field or a local field, it follows from class field theory that $\sqrt{\dim_{E} D}$ is equal to the order of $[D]$ in $\Br E$ (see e.g. \cite{pierce:algebras}~Theorem~18.6).

Now apply Lemma~\ref{lem:tits-brauer-order}.
\end{proof}

The following theorem is obtained by combining Corollaries~\ref{cor:product-bound} and~\ref{cor:irred-rep-multiplicity}.

\begin{theorem} \label{thm:mt-triple-bound-noncommutative}
Let $(G, \rho, \Psi)$ be a pure weak Mumford--Tate triple of weights $\zeroone$ defined over a number field or a local field $F$.
If the $F$-irreducible components of $\rho$ are pairwise non-isomorphic, then
\[ \rk G + \log_2 u(G) \geq \log_2 \dim \rho + 1. \]
\end{theorem}

If $A$ is a complex abelian variety, $G$ its Mumford--Tate group and $\rho$ the associated representation, then $\End \rho = \End A \otimes_\Z \Q$ so the hypothesis that the simple abelian subvarieties of $A$ are pairwise non-isogenous implies that the irreducible components of $\rho$ are pairwise non-isomorphic.
Hence Theorem~\ref{thm:mt-triple-bound-noncommutative}, together with the bounds for $u(G)$ in section~\ref{subsec:uG}, implies Theorem~\ref{thm:main-bound-noncommutative} for the Mumford--Tate group.

\subsection{Multiplicities in $\ell$-adic representations}

Let $A$ be an abelian variety over a number field whose simple abelian subvarieties are pairwise non-isogenous.
Let $G_\ell$ be the $\ell$-adic monodromy group of $A$ and $\rho_\ell$ the associated $\ell$-adic representation.
We shall show that the multiplicities of irreducible components of $\rho_\ell \otimes_{\Q_\ell} \bar\Q_\ell$ are bounded above by $u(G_\ell^\circ)$ and hence prove Theorem~\ref{thm:main-bound-noncommutative}.

By Faltings' Theorem, if $B$ and $B'$ are non-isogenous simple abelian varieties, then the associated $\ell$-adic representations have no common subrepresentations.
Hence it will suffice to suppose that $A$ is simple.

By Faltings' Theorem, $\End \rho_\ell = \End A \otimes_\Z \Q_\ell$.
This implies that the multiplicities of absolutely irreducible components of $\rho_\ell \otimes_{\Q_\ell} \bar\Q_\ell$ are independent of $\ell$.
We will use results of Serre and Pink to show that $u(G_\ell^\circ)$ is also independent of $\ell$, and then we can consider all $\ell$ at once to show that the multiplicities are bounded above by $u(G_\ell^\circ)$.

\begin{lemma} \label{lem:u-indep-of-ell}
$u(G_\ell^\circ)$ is independent of $\ell$.
\end{lemma}

\begin{proof}
Let $\ell$, $\ell'$ be any two rational primes.
Via $\rho_\ell$, we view $G_\ell^\circ$ as a subgroup of $\GL_{2g,\Q_\ell}$.

For a finite place $v$ of $K$, let $T_v$ be the Frobenius torus of $A$ in the sense of Serre~\cite{serre:frob-tori}.
Serre showed that we can choose $v$ such that
$T_{v,\Q_\ell}$ is $\GL_{2g,\Q_\ell}$-conjugate to a maximal torus of $G_\ell^\circ$,
and such that the analogous property holds for $\ell'$.

Hence we get maximal tori $T_{v,\ell}$ of $G_\ell$ and $T_{v,\ell'}$ of $G_\ell'$ together with an isomorphism $\Lambda(T_{v,\ell}) \cong \Lambda(T_v) \cong \Lambda(T_{v,\ell'})$.
Furthermore, under this isomorphism,
the formal character of $\rho_\ell$ corresponds to the formal character of $\rho_{\ell'}$.

As observed by Larsen-Pink~\cite{larsen-pink:inv-dims}, the formal character of a faithful irreducible representation of a reductive group determines the root lattice $\Lambda_0$.
Hence $\Lambda/\Lambda_0(G_\ell^\circ) \cong \Lambda/\Lambda_0(G_{\ell'}^\circ)$ so $u(G_\ell^\circ) = u(G_{\ell'}^\circ)$.
\end{proof}

We will also need the following lemma on pure weak Mumford--Tate triples.
Let $(G, \rho, \Psi)$ be a pure weak Mumford--Tate triple.
Because it is pure, there is a cocharacter $\mu_0$ of $G$ such that $\rho \circ \mu_0(z) = z.\id$.
Let $H$ be the identity component of $\ker \det \rho \subset G$.
Then the quasi-cocharacter space of a maximal torus $T$ splits as
\[ \left( \Hom(\G_m, T \cap H) \otimes_\Z \Q \right) \oplus \Q.\mu_0. \tag{*} \label{eqn:hodge-group-decomposition} \]

\begin{lemma} \label{lem:hodge-group-weights}
Let $(G, \rho, \Psi)$ be a pure weak Mumford--Tate triple of weights $0$ and $1$, with multiplicities $g_0$ and $g_1$ respectively.
Choose $\mu \in \Psi$ and let $T$ be a maximal torus of $G$ containing the image of $\mu$.
Suppose that $\mu$ splits as $\mu_H + r \mu_0$ in the decomposition \eqref{eqn:hodge-group-decomposition}.
Then for all characters $\omega$ in $\rho_{|T}$,
\[ \innerprod{\omega}{\mu_H} = \frac{g_0}{g_0+g_1} \text{ or } \frac{-g_1}{g_0+g_1}. \]
\end{lemma}

\begin{proof}
By the definition of $\mu_0$, $\innerprod{\omega}{\mu_0} = 1$ for every character $\omega$ in $\rho_{|T}$.
Hence
\[ \innerprod{\det \rho}{\mu_0} = \dim \rho = g_0 + g_1. \]

Because $\det \rho$ is trivial on~$H$, $\innerprod{\det \rho}{\mu_H} = 0$.
Therefore
\[ \innerprod{\det \rho}{\mu} = \innerprod{\det \rho}{r \mu_0} = r(g_0 + g_1). \]
On the other hand, 
\[ \innerprod{\det \rho}{\mu} = g_0.0 + g_1.1 = g_1 \]
so $r = g_1/(g_0 + g_1)$.
Combining with $\innerprod{\omega}{\mu} = 0$ or $1$ gives the result.
\end{proof}

\begin{proposition} \label{prop:l-adic-multiplicity}
Let $A$ be a simple abelian variety defined over a number field, and $G_\ell$ its $\ell$-adic monodromy group.
The multiplicity of every absolutely irreducible component of $\rho_\ell \otimes_{\Q_\ell} \bar{\Q_\ell}$ divides $u(G_\ell^\circ)$.
\end{proposition}

\begin{proof}
Let $D = \End A \otimes_{\Z} \Q$ be the endomorphism algebra of $A$,
and let $E$ be the centre of $D$.
Let $m^2 = \dim_{E} D$.

By Faltings' Theorem, $\End \rho_\ell = D \otimes_{\Q} \Q_\ell$.
This is a product of simple algebras, each of dimension $m^2$ over its centre.
So every absolutely irreducible component of $\rho_\ell \otimes_{\Q_\ell} \bar{\Q_\ell}$ has multiplicity $m$,
and it will suffice to show that $m$ divides $u(G_\ell^\circ)$.

There are two cases: $E$ is totally real or a CM field.

\paragraph*{Case 1.  $E$ is totally real}
In this case the Albert classification of endomorphism algebras of abelian varieties implies that $m \leq 2$, so it will suffice to show that $2$ divides $u(G_\ell^\circ)$.

Let $H_\ell$ be the identity component of $\ker \det \rho_\ell \subset G_\ell$ (in other words, the $\ell$-adic analogue of the Hodge group).
By \cite{tankeev:hodge-semisimple}~Lemma~1.4, the condition that $E$ is totally real implies that the Hodge group of $A$ is semisimple and by \cite{silverberg-zarhin}~Theorem~3.2 this implies that $H_\ell$ is semisimple.
Hence $H_\ell$ is the derived group of $G_\ell^\circ$.

Let $\mu$ be a weak Hodge cocharacter of $G_\ell$ in the sense of \cite{pink:mt}~Definition~3.2 and let $T$ be a maximal torus of $G_\ell$ containing the image of $\mu$.
Then $\rho_\ell \circ \mu$ has weights $0$ and $1$ each with multiplicity $\dim A$, so by Lemma~\ref{lem:hodge-group-weights},
\[ \innerprod{\omega}{\mu_H} = \pm \half \]
for all characters $\omega$ in $\rho_{\ell|T}$, where $\mu_H$ is the component of $\mu$ in the quasi-cocharacter space of $T \cap H_\ell$.

Now $\innerprod{-}{\mu}$ takes integer values on the roots of~$G_\ell^\circ$.
Since $\mu_0$ is orthogonal to the roots, the same is true for $\innerprod{-}{\mu_H}$.
Because $H_\ell$ is semisimple, it is equal to the derived group of $G_\ell^\circ$,
so $\mu_H$ is orthogonal to all characters which vanish on $T \cap G_\ell^{\circ \der}$.
Hence $\innerprod{-}{\mu_H}$ takes integer values on $\Lambda_0(G_\ell^\circ)$.

So in order for $\innerprod{\omega}{\mu_H}$ to have denominator $2$, the order of $\omega$ in $\Lambda(G_\ell) / \Lambda_0(G_\ell^\circ)$ must be even.
Therefore $u(G_\ell^\circ)$ is divisible by~$2$.

\paragraph*{Case 2.  $E$ is a CM field}
For each place $\lambda$ of $E$, let $E_\lambda$ denote the completion of $E$ at $\lambda$.
Then $D_\lambda = D \otimes_{E} E_{\lambda}$ is a matrix ring over a division algebra with centre $E_{\lambda}$.
Let $m_\lambda$ be the order of $[D_\lambda]$ in $\Br E_\lambda$.
By the Albert--Brauer--Hasse--Noether theorem (\cite{pierce:algebras} Theorem~18.5), the map $[D] \mapsto ([D_\lambda])$ is an injection
\[ \Br E \to \bigoplus_\lambda \Br E_\lambda \]
so $m$ is the lowest common multiple of the $m_\lambda$.
So it suffices to show that $m_\lambda$ divides $u(G_\ell^\circ)$ for every place $\lambda$.

Since $E$ is a CM field, all its archimedean places have trivial Brauer group, so we need only consider non-archimedean places.
Let $\lambda$ be a non-archimedean place of $E$ and $\ell'$ its residue characteristic.
Then
\[ \End \rho_{\ell'} = D \otimes_{\Q} \Q_{\ell'}
   = D \otimes_{E} \left( \prod_{\lambda'|\ell'} E_{\lambda'} \right)
   = \prod_{\lambda'|\ell'} D_{\lambda'}. \]
Hence $\rho_{\ell'}$ has a $\Q_{\ell'}$-irreducible subrepresentation with endomorphism algebra~$D_{\lambda}$.

So by Lemma~\ref{lem:tits-brauer-order}, $m_\lambda$ divides $u(G_{\ell'}^\circ)$,
and this is equal to $u(G_\ell^\circ)$ by Lemma~\ref{lem:u-indep-of-ell}.
\end{proof}

Theorem~\ref{thm:main-bound-noncommutative} follows from Corollary~\ref{cor:product-bound}, Proposition~\ref{prop:l-adic-multiplicity} and the bounds for $u(G)$ in section~\ref{subsec:uG}.

\subsection{Bounds for $u(G)$} \label{subsec:uG}

\begin{definition}
Let $g(n)$ be the maximum value of $\lcm ( a_i )$ where $a_i$ are positive integers satisfying $\sum a_i = n$.  (This is Landau's function.)

Let $g_1(n)$ be the maximum value of $\lcm ( a_i )$ where $a_i$ are integers greater than $1$ satisfying $\sum (a_i - 1) = n$.

For $n \geq 2$, let
\[ \alpha(n) = \frac{\log_2 g_1(n)}{\sqrt{n \log n}}. \]
\end{definition}

\begin{lemma} \label{lem:u-to-landau}
For any reductive group $G$, $u(G) \leq g_1(\rk G)$.
\end{lemma}

\begin{proof}
Let $\Phi_i$ (for $i \in I$) be the simple components of the root system of~$G$.

The group $\Lambda/\Lambda_0$ is a subgroup of the product of the fundamental groups of the $\Phi_i$.
So $u(G)$ divides the lowest common multiple of the exponents of these fundamental groups.

Let $e_i$ be the exponent of the fundamental group of $\Phi_i$.
Then $e_i \leq \rk \Phi_i + 1$ for all $i$ (by the classification of simple root systems),
and so $\sum_i (e_i - 1) \leq \rk G$.

By the definition of $g_1$,
\[ u(G) \leq g_1 \left(\sum_{i \in I} (e_i - 1) \right) \]
and this is less than or equal to $g_1(\rk G)$ because $g_1$ is nondecreasing.
\end{proof}

\begin{corollary} \label{cor:alpha-limit}
$\alpha(n) \to \frac{1}{\log 2}$ as $n \to \infty$ and $\alpha(n) < 2$ for all $n \geq 2$.
\end{corollary}

\begin{proof}
We use two results on the size of $g(n)$: Landau's asymptotic result~\cite{landau:asymptotic}
\[ \frac{\log_e g(n)}{\sqrt{n \log n}} \to 1 \text{ as } n \to \infty \]
and Massias' bound~\cite{massias:landau}
$$ \log_e g(n) < 1.05314 \sqrt{n \log n} \text{ for all } n \geq 2. $$

We note that $g(n) \leq g_1(n) \leq g\nfsqr$ since any set of distinct positive integers satisfying
$\sum_i (a_i - 1) = n$ will satisfy $\sum_i a_i \leq n + \lfloor \sqrt{2n} \rfloor$.

Let
\[ f(x) = \frac{(x+\sqrt{2x}) \log(x+\sqrt{2x})}{x \log x}. \]

Since $f(x) \to 1$ as $x \to \infty$, we conclude that $\alpha(n) \to \frac{1}{\log 2}$.

Likewise by Massias' bound
\[ \alpha(n)
   \leq \frac{\log_e g \nfsqr}{\log 2 \, \sqrt{n \log n}}
   < \frac{1.05314 \sqrt{f(n)}}{\log 2}
   \leq \frac{1.05314 \sqrt{f(9)}}{\log 2}
   < 2 \]
for $n \geq 9$ since $f(x)$ is decreasing for $x > 1$.

Manual calculation shows that $\alpha(n) < 2$ for $2 \leq n \leq 8$.
\end{proof}

\section{Some examples} \label{sec:examples}

In this section, we will give three examples of families of abelian varieties with commutative endomorphism ring for which Theorem~\ref{thm:main-bound-commutative} is sharp.
Note that any abelian variety for which the rank of the Mumford--Tate group is equal to the bound of Theorem~\ref{thm:main-bound-commutative} necessarily satisfies the Mumford--Tate conjecture, because the rank of the $\ell$-adic monodromy groups are less than or equal to that of the Mumford--Tate group but satisfy the same lower bound.
Hence the examples we give show that the bound is sharp for the $\ell$-adic monodromy group as well as for the Mumford--Tate group.

We also give one family of simple abelian varieties with noncommutative endomorphism ring for which the Mumford--Tate group has rank $n$ and the dimension $g$ satisfies $\log_2 g = n + \half \log_2 n + \mathrm{O}(1)$.  This shows that the bound in Theorem~\ref{thm:main-bound-noncommutative} cannot be improved to $n \geq \log_2 g + \mathrm{O}(1)$.
Because we have not calculated the exact lower bound in the noncommutative case we cannot deduce that these varieties satisfy the Mumford--Tate conjecture purely from the rank bound.
But for the examples constructed here, we can show that they satisfy the Mumford--Tate conjecture by using \cite{pink:mt}~Proposition~4.3.

\subsection{Examples with commutative endomorphism ring}

\paragraph*{Example 1: Complex multiplication.}

Let $F$ be a totally real field such that $[F:\Q] = n-1$.
By \cite{shimura:canonical-models}~Theorem~1.10,
there is an imaginary quadratic extension $K$ of $F$ such that for every CM type $(K, \Phi)$,
the reflex type $(K', \Phi')$ satisfies $[K':\Q] = 2^{n-1}$.
Such a CM type is primitive.

Let $A$ be a complex abelian variety corresponding to the CM type $(K', \Phi')$.
Then the Mumford--Tate group $M$ is a torus,
isomorphic to the image of the homomorphism $\Res_{K/\Q} \G_m \to \Res_{K'/\Q} \G_m$
induced by the reflex norm $K^\times \to K'^\times$.

This image has rank at most $[K:\Q] + 1 = n + 1$.
But $\dim A = 2^{n-1}$ so by Theorem~\ref{thm:main-bound-commutative}, $\rk M \geq n + 1$.
So in fact $\rk M = n + 1 = \log_2 \dim A + 2$.

The endomorphism ring of $A$ is the field $K'$.

\paragraph*{Example 2: Spin group.}

This example generalises the Kuga-Satake construction of an abelian variety attached to a polarised $K3$ surface~\cite{kuga-satake:construction}.

Let $n$ be a positive integer congruent to $1$ or $2 \bmod 4$.
Let $W$ be a $\Q$-vector space of dimension $2n+1$,
and let $Q$ be the quadratic form
\[ Q(x) =  x_1^2 + x_2^2 - x_3^2 - \dotsb - x_{2n+1}^2 \]
of signature $(2, 2n-1)$.
The even Clifford algebra $C^+(W, Q)$ is isomorphic to $\operatorname{M}_{2^n}(\Q)$,
and so it has a unique faithful irreducible $\Q$-representation of dimension $2^n$,
called the spin representation.

Let $M$ be the Clifford group
\[ \operatorname{GSpin}(W, Q) = \{ x \in C^+(W, Q) \mid xWx^{-1} \subseteq W \}. \]
This is a reductive group of rank $n+1$, with root system~$B_n$ and centre $\G_m$.
Let $\rho : M \to \GL(V)$ be the spin representation of $M$.
This is an absolutely irreducible representation of dimension~$2^n$.

Let $\{ e_1, e_2 \}$ be an orthonormal basis for the positive definite subspace of $W$.
The homomorphism $\varphi : \C^\times \to M(\R)$ given by
\[ \varphi(a+ib) = a+b e_1 e_2 \]
defines a Hodge structure on $V$ of type $\{ (0,-1), (-1,0) \}$.
The conditions on $n \bmod 4$ and on the signature of $W$ ensure that this Hodge structure is polarisable.

Because $M^\der$ is almost simple, replacing $\varphi$ by a generic $M(\R)$-conjugate gives a Hodge structure whose Mumford--Tate group is $M$.
Let $A$ be a complex abelian variety corresponding to such a Hodge structure.
It has dimension $2^{n-1}$ and endomorphism algebra $\Q$, and its Mumford--Tate group has rank $n+1$.

\paragraph*{Example 3: Product of copies of $\operatorname{SL}_2$.}

This example generalises the example of Mumford~\cite{mumford:families} of a family of abelian varieties of dimension $4$ with Mumford--Tate group $M$ such that $M_\C$ is isogenous to $\G_m \times (\operatorname{SL}_2)^3$.

Let $n$ be an odd positive integer,
and $F$ a totally real number field of degree $n$.
Let $D$ be a quaternion algebra over $F$ such that:
\begin{enumerate}[(i)]
\item $\operatorname{Cor}_{F/\Q} D$ is split over $\Q$, i.e.\ is isomorphic to $\operatorname{M}_{2^n}(\Q)$.
\item $D$ is split at exactly one real place of $F$.
\end{enumerate}

Let $M$ be the $\Q$-algebraic group $M(A) = \{ x \in (D \otimes A)^\times | x \bar{x} \in A^\times \}$
(where $\bar{x}$ is the standard involution of $D$).
By condition~(ii), $M_\R$ is isomorphic to
\[ \left( \G_{m,\R} \times \operatorname{SL}_{2,\R} \times \operatorname{SU}_2^{n-1} \right)
   / \left\{ (\eps_0, \eps_1, \dotsc, \eps_n) \mid \eps_i \in \{ \pm 1 \}, \eps_0\eps_1\dotsb\eps_n = 1 \right\}. \]

By condition~(i), $M$ has a faithful irreducible $\Q$-representation $\rho$ of dimension $2^n$.
Then $\rho \otimes_\Q \C$ is isomorphic to the tensor product of the standard $1$-dimensional representation of $\G_{m,\C}$ with the standard $2$-dimensional representation of each factor $\operatorname{SL}_{2,\C}$.

Let $\varphi : \C^\times \to M(\R)$ be the homomorphism
\begin{align*}
\varphi(a+ib) = & \fullsmallmatrix{a}{b}{-b}{a} \text{ in } \GL_2 \isom \left( \G_m \times \operatorname{SL}_2 \right) / \{ \pm 1 \}
\\		& \text{ and trivial in the } \operatorname{SU_2} \text{ factors}.
\end{align*}
Then $\rho \circ \phi$ defines a Hodge structure of type $\{ (0,-1), (-1,0) \}$.
By condition~(ii), this Hodge structure is polarisable.

Again $M^\der$ is almost $\Q$-simple,
so replacing $\varphi$ by a generic element of its $M(\R)$-conjugacy class gives a Hodge structure with Mumford--Tate group equal to $M$.
An abelian variety corresponding to such a Hodge structure will have dimension $2^{n-1}$, endomorphism algebra $\Q$ and Mumford--Tate group of rank $n+1$.

\subsection{An example with large multiplicity}

Let $n$ be an odd integer and $r = (n-1)/2$.
We will construct a simple abelian variety of dimension $g(n) = n \binom{n}{r}$ whose Mumford--Tate group is a $\Q$-form of $\GL_n$.
The Mumford--Tate representation is isomorphic over $\C$ to the sum of $2n$ copies of the $r$-th exterior power of the standard representation.
By Stirling's formula $\log_2 g(n) = n + \half \log_2 n + \mathrm{O}(1)$.

Let $K$ be an imaginary quadratic field, and $D$ a central division algebra over $K$ of dimension $n^2$ with an 
involution $*$ of the second kind.
The $\Q$-algebraic groups
\begin{align*}
H(A) &= \{ d \in (D \otimes_\Q A)^\times \mid dd^* = 1 \},	\\
G(A) &= \{ d \in (D \otimes_\Q A)^\times \mid dd^* \in A^\times \}
\end{align*}
are $\Q$-forms of $\operatorname{SL}_n$ and $\operatorname{GL}_n$.
By choosing $*$ appropriately, we may suppose that $H_\R$ is the unitary group of a Hermitian form of signature $(1, n-1)$.

We can view $D$ as a $K$-irreducible representation of $H_K$.
Over $\C$, $D_\C$ is isomorphic to the sum of $n$ copies of the standard representation of~$\operatorname{SL}_n$, so its highest weight is $\varpi_1$.
The endomorphism ring of this representation is $D^\op$, so
\[ \alpha_{H,K}(\varpi_1) = [D] \]
for Tits' homomorphism $\alpha_{H, K} : \Lambda^\Gamma \to \Br K$.

Let $r = (n-1)/2$ and let $\tilde{D}$ be the central division algebra over $K$ such that $[\tilde{D}] = [D]^r$ in $\Br K$.
Now $[D]$ has order $n$ in $\Br K$.
Since $r$ and $n$ are coprime, $[\tilde{D}]$ also has order $n$ and $\tilde{D} \otimes_K \C \cong M_n(\C)$.

Let $\tilde{\rho}$ be the $K$-irreducible representation of $H_K$ with highest weight $\varpi_r$.
We know that $\varpi_r \equiv r\varpi_1$ modulo the roots of $H_K$, so $\alpha_{H,K}(\varpi_r) = [D]^r = [\tilde{D}]$.
Hence $\tilde{\rho}$ has endomorphism ring $\tilde{D}^\op$, so $\tilde{\rho}_\C$ is the sum of $n$ copies of an irreducible representation of $\operatorname{SL}_n$.
This irreducible representation is the $r$-th exterior power of the standard representation,
so $\dim_K \tilde{\rho} = n\binom{n}{r}$.

If $\lambda I$ is a scalar matrix in $H(\C)$, then $\tilde{\rho}_\C(\lambda I)$ is multiplication by $\lambda^r$.
So we can extend $\tilde{\rho}$ to a representation of $G_K$ by letting each scalar matrix $\lambda I$ act as multiplication by $\lambda^r$.

Let $\rho = \Res_{K/\Q} \tilde{\rho}$.
This is a $\Q$-irreducible representation of $G$ of dimension $2n\binom{n}{r}$.
We have $\ker \rho = \mu_r$ so $\rho$ factorises through $M = G/\mu_r$,
and the resulting representation of $M$ is faithful.

In order to specify the Hodge structure, we will first define $\varphi' : \mathbb{C}^\times \to G(\mathbb{R})$ as follows:
recall that $H_\R$ is the unitary group of a Hermitian form $\Psi$ of signature $(1, n-1)$.
Then let $\phi'(z)$ act as $z^r/\bar{z}^{r-1}$ on the $1$-dimensional space where $h$ is positive definite and as $\bar{z}$ on the $(n-1)$-dimensional space where $h$ is negative definite.

Then $\rho \circ \varphi'$ has weights $z^r$ and $\bar{z}^r$.
Because $\rho$ is faithful as a representation of~$M$, it follows that there is a homomorphism
$\varphi : \C^\times \to M(\R)$ whose $r$-th power is~$\varphi'$.
Then $(M, \rho, \varphi)$ defines a $\Q$-Hodge structure of type $\{ (-1,0), (0,-1) \}$.
The Hermitian form $\Psi$ induces a polarisation of this Hodge structure.

Once again, $M^\der$ is almost simple, so replacing $\varphi$ by a generic $M(\R)$-conjugate gives a Hodge structure with Mumford--Tate group $M$.
A corresponding abelian variety will have endomorphism algebra $\tilde{D}^\op$ and dimension
$g = n \binom{n}{r}$.

We shall confirm that this variety satisfies the Mumford--Tate conjecture.
Let $\sigma$ be an absolutely irreducible component of $\rho \otimes_\Q \bar{\Q_\ell}$.
Then $(G \times_\Q \bar{\Q_\ell}, \sigma)$, with a suitable set of cocharacters, form a weak Mumford--Tate triple of weights $\{ 0, 1 \}$.
By Faltings' theorem, the restriction of $\sigma$ to $G_{\ell, \bar{\Q_\ell}}$ must remain irreducible, where $G_\ell$ is the $\ell$-adic monodromy group.
It also is part of a weak Mumford--Tate triple of weights $\{ 0, 1 \}$.
But our $(G \times_\Q \bar{\Q_\ell}, \sigma)$ is in the fourth column of \cite{pink:mt} Table~4.2: type A with $\sigma$ not the standard representation.
Hence according to Pink's Proposition~4.3, $G_\ell = G \times_\Q \Q_\ell$.

\backmatter
\bibliographystyle{smfalpha}
\bibliography{mt}

\end{document}

%% file: preamble.tex
\usepackage[utf8x]{inputenc}
\usepackage[T1]{fontenc}
\usepackage{lmodern}

\usepackage{amsmath}
\usepackage{amssymb}
\usepackage{amsthm}
\usepackage[all]{xy}
\usepackage{tikz}
\usepackage{enumerate}
\usepackage{multirow}
\usepackage[bookmarksnumbered,colorlinks,linkcolor=black,citecolor=black,urlcolor=black]{hyperref}

\newcommand{\C}{\mathbb{C}}

\newcommand{\G}{\mathbb{G}}
\newcommand{\N}{\mathbb{N}}
\newcommand{\Q}{\mathbb{Q}}
\newcommand{\R}{\mathbb{R}}
\newcommand{\Z}{\mathbb{Z}}

\newcommand{\eps}{\varepsilon}

\newcommand{\Aut}{\mathop{\mathrm{Aut}}}
\newcommand{\Br}{\mathop{\mathrm{Br}}}

\newcommand{\der}{\mathrm{der}}

\newcommand{\End}{\mathop{\mathrm{End}}}

\newcommand{\Gal}{\mathop{\mathrm{Gal}}}
\newcommand{\GL}{\mathop{\mathrm{GL}}\nolimits}
\newcommand{\Hom}{\mathop{\mathrm{Hom}}}
\renewcommand{\Im}{\mathop{\mathrm{Im}}}

\newcommand{\Res}{\mathop{\mathrm{Res}}\nolimits}

\newcommand{\abs}[1]{\lvert #1 \rvert}

\newcommand{\op}{\mathrm{op}}
\newcommand{\id}{\mathrm{id}}

\newcommand{\innerprod}[2]{\langle #1, #2 \rangle}

\newcommand{\isom}{\cong}

\newcommand{\fullsmallmatrix}[4]{\bigl( \begin{smallmatrix} #1 & #2 \\ #3 & #4 \end{smallmatrix} \bigr)}

\newcommand{\Galabs}[1]{\Gal(\bar{#1}/#1)}

\newcommand{\zeroone}{\left\{ 0, 1 \right\}}

\newcommand{\half}{\tfrac{1}{2}}

\newcommand{\defterm}[1]{\textit{#1}}

\newtheorem{lemma}{Lemma}[section]
\newtheorem{proposition}[lemma]{Proposition}
\newtheorem{theorem}[lemma]{Theorem}
\newtheorem{corollary}[lemma]{Corollary}

\newtheorem*{proposition*}{Proposition}
\newtheorem*{corollary*}{Corollary}

\theoremstyle{definition}
\newtheorem*{definition}{Definition}